\documentclass[11pt]{article}

\usepackage{epsfig}
\usepackage{amsmath,amssymb,amsthm,amsxtra}
\usepackage{amsfonts}
\usepackage{graphicx} 
\usepackage{mathrsfs}
\usepackage{setspace} 

\newcommand{\G}{\Gamma}
\newcommand{\ep}{\varepsilon}

\usepackage{float}
\usepackage[margin=0pt,font=small,labelfont=bf,textfont=it]{caption}

\theoremstyle{plain}

\newtheorem{Theorem}{Theorem}[section]	

\newtheorem{Lemma}[Theorem]{Lemma}
\newtheorem{Definition}[Theorem]{Definition}

\title{On the $\Gamma$-limit for a non-uniformly bounded sequence of
two phase metric functionals} 

\author{Hartmut Schwetlick\footnote{E-Mail: schwetlick@maths.bath.ac.uk}, Daniel C. Sutton\footnote{E-Mail: d@csutton.eu} and Johannes Zimmer\footnote{E-Mail: zimmer@maths.bath.ac.uk}\\Department of Mathematical Sciences\\
University of Bath\\
Bath BA2 7AY, U.K.}

\date{December 2013}

\begin{document}

\maketitle

\begin{abstract}
In this study we consider the $\G$-limit of a highly oscillatory Riemannian metric length functional as its period tends to 0. The metric coefficient takes values in either $\{1,\infty\}$ or $\{1,\beta \ep^{-p}\}$ where $\beta,\ep > 0$ and $p \in (0,\infty)$. We find that for a large class of metrics, in particular those metrics whose surface of discontinuity forms a differentiable manifold, the $\G$-limit exists, as in the uniformly bounded case. However, when one attempts to determine the $\G$-limit for the corresponding boundary value problem, the existence of the $\G$-limit depends on the value of $p$. Specifically, we show that the power $p=1$ is critical in that the $\G$-limit exists for $p < 1$, whereas it ceases to exist for $p \geq 1$. The results here have applications in both nonlinear optics and the effective description of a Hamiltonian particle in a discontinuous potential.
\end{abstract}

\newpage

\section{Introduction}
Let $\emptyset \ne \Omega_g \subset [0,1]^d$ be an open and path connected set with Lipschitz boundary. Suppose further that $\mathbb{R}^d \setminus (\Omega_g + \mathbb{Z}^d)$ and $\partial \Omega_g$ are path connected. We study the sequence of functionals
\begin{equation}
\label{s1:eq:functional}
F_{p,\ep}(u) := \int_0^1 a_{p,\ep}\left(\frac{u}{\ep}\right)\|u'\| \, \text{d} \tau, \quad u \in W^{1,\infty}\left((0,1)\right),
\end{equation}
where \begin{equation}
a_{p,\ep}(x) := 
\begin{cases}
\beta \ep^{-p} & \text{ if } x \in \Omega_g + \mathbb Z^d,\\
1 & \text{otherwise,}
\end{cases}
\end{equation}
for $\beta > 0$ and $p \in (0,\infty)$. We will also consider a limiting case where $\left. a_{p,\ep} \right|_{p=\infty} : = \infty$ on $\Omega_g + \mathbb Z^d$ and $1$ otherwise.

Functionals of the form \eqref{s1:eq:functional} naturally arise in the study of Riemannian geometry and nonlinear optics. In the case of nonlinear optics the values of $a_{p,\ep}$ describe the \emph{opacity} of the material. Our motivation is to study the effective dynamics of a Hamiltonian particle. In classical mechanics the motion of a particle with unit mass is given by
\begin{equation}
\label{s1:eq:n2l}
\frac{\text{d}^2x}{\text{d} t^2} = -\nabla V(x),
\end{equation}
where $V$ is the potential energy function. An alternative description of the motion of such a particle is given by the \emph{Maupertuis principle} \cite{ba,arnold97a,marsden99a,schwetlick09a}. The Maupertuis principle states that, provided $V$ is sufficiently smooth, that the solutions of \eqref{s1:eq:n2l} are critical points of the functional
\begin{equation}
\int_0^1 \sqrt{2\left(E - V \left(\frac{u}{\ep}\right)\right)}\|u'\| d \tau,
\end{equation}
up to reparameterisation. The functional \eqref{s1:eq:functional} could therefore be interpreted as a model for the motion of a particle in a discontinuous periodic potential. The discontinuities and large metric values in $\Omega_g + \mathbb Z^d$ model positional constraints on the motion of the particle. We are interested in approximating minimum points of the functional \eqref{s1:eq:functional} for small $\ep > 0$ in the weak topology of $W^{1,\infty}\left((0,1)\right)$. Determining the effective description of a Hamiltonian particle in a discontinuous potential has been posed as a problem in \cite{lions88a}. By determining an effective limit for the functional \eqref{s1:eq:functional} we begin to develop an insight into the average motion of a particle in these potentials. This approach provides an alternative to determining an effective description for the Hamilton-Jacobi equations that also describe the motion of a Hamiltonian particle. The homogenisation of the Hamilton-Jacobi equation was first studied in \cite{lions88a}.

Effective descriptions for Riemannian metrics that satisfy uniform growth conditions have been studied in \cite{acerbi84a,braides02a,braides98a,braides92a,e91a}. A proof that the Riemannian metrics are dense with respect to $\G$-convergence in the Finlser metrics can be found in \cite{braides02b}. Examples of Finsler metrics obtained as the $\G$-limit of a sequence of two phase Riemannian metrics may be found in \cite{amar09a,braides98a,buttazzo01a,concordel97a,c03,oberman09a,sutton13a,suttontese}. The difference between the work in \cite{amar09a,braides98a,buttazzo01a,oberman09a,sutton13a,suttontese} and here is that our metrics are not uniformly equivalent, in $\ep$, to the Euclidean distance on $\mathbb{R}^d$. Specifically, for a fixed $p$, there does not exist $\alpha, \beta > 0$ such that
\begin{equation}
\int_0^1 \alpha \| u'(\tau) \| \, \text{d}\tau \leq F_{p,\ep}(u) \leq \int_0^1 \beta \| u'(\tau) \| \, \text{d}\tau
\end{equation}
for all $\ep > 0$. Consequently, in this setting, the standard theory does not apply. 

In this paper we will focus on the computation of the $\G$-limit of \eqref{s1:eq:functional}. Specifically we will show that the $\G$-limit of \eqref{s1:eq:functional} on $W^{1,\infty}\left((0,1)\right)$ exists. Furthermore, we show that this limit can be described by the $\G$-limit of a sequence of uniformly bounded length functionals. We will continue by studying the $\G$-limit of the corresponding boundary value problem. In the literature one usually finds that the $\G$-limit for the boundary value problem follows from the $\G$-limit for the unconstrained problem \cite{braides02a,braides98a}. The novel observation that we make here is that in the absence of a uniform growth condition the $\G$-limit may no longer exist. The existence of the $\G$-limit depends on the value for $p$. We show that the value $p=1$ is critical in the sense that for $p < 1$ the $\G$-limit for the boundary value problem exists, whereas for $p \geq 1$ the $\G$-limit does not exist. For $p < 1$ we show that the minimum values of $F_{p,\ep}$ converge and then prove an analogue of \cite{buttazzo01a}, that is, the length functionals $\G$-converge if, and only if, the induced distance functions converge locally uniformly. Contrary, for $p \geq 1$, and with an additional geometric assumption, we show that the $\G$-limit fails to exist. Specifically we present a counter example, that is the existence of two sequences of $\ep$ such that minimisers of \eqref{s1:eq:functional} converge to different limits. By modifying the ideas in \cite{buttazzo01a} we can study the $\G$-convergence of $F_{p,\ep}$. This work is a natural continuation of the examples at the end of \cite{buttazzo01a} to a more general class of problems. Furthermore, this paper highlights the difficulties one may encounter when determining the $\G$-limit of a sequence of non-uniformly bounded functionals. In particular, one has too be very careful when imposing microscopic positional constraints when determining the effective dynamics of a Hamiltonian particle. For a given value for $\beta > 0$, $p \in [0,\infty]$ and $\ep$ sufficiently small, the value of $a_{p,\ep}$ will prevent the particle from entering $\Omega_g + \mathbb{Z}^d$, therefore asymptotically we will always impose a positional constraint on the particle. Our results show that one has to take care in choosing $p$, and hence scale the potential wells in $\ep$, to ensure that the limit problem exists.

\subsection*{Acknowledgements}

DCS was funded by an EPSRC DTA. The authors would like to thank Karsten Matthies and Grigoris Pavliotis for their comments on an earlier draft of this work.

\section{High opacity coefficients}
In this section the notion of a \emph{high opacity coefficient} is formulated. The high opacity coefficient provides a lower bound on the opacity of inclusions that prevents geodesics entering them. Throughout this paper a set $\emptyset \ne \Omega_g \subsetneq [0,1]^d$ is an \emph{admissible inclusion} if it is path connected, open and has a Lipschitz boundary. Furthermore we will assume that $\Omega_g$ has the property that $\Omega_w := \mathbb R^d \setminus \left( \Omega_g + \mathbb Z^d \right)$ and $\partial \Omega_g$ are path connected. For notational convenience set $\mathscr A (x,y) := \{ u \in W^{1,\infty}\left((0,1)\right) \colon u(0) = x, u(1) = y \}$; where $W^{1,\infty}\left((0,1)\right)$ is the space of all Lipschitz curves on $(0,1)$ taking values in $\mathbb R^d$. For $E \subset \mathbb R^d$ nonempty, define the mapping $d_E \colon E \times E \times [0,\infty) \rightarrow \mathbb R$ by
\begin{equation*}
d_E(x,y;\beta) := \inf_{u \in \mathcal A(x,y)} \left\{ \int_0^1 \beta \|u'(\tau) \| \, \text{d}\tau \colon u(\tau) \in E \; \forall \tau \in (0,1) \right\}.
\end{equation*}
To simplify notation further we set $d_E(x,y):=d_E(x,y;1)$.

\begin{Definition}[High Opacity Coefficient]
Let $\Omega_g$ be an admissible inclusion. A high opacity coefficient for the set $\Omega_g$ is a number $\lambda \in (0, \infty)$ such that for all $x,y \in \partial \Omega_g$ and $\beta > \lambda$,
\begin{equation}\label{3:hcc}
d_{\partial \Omega_g}(x,y) < d_{\Omega_g}(x,y;\beta).
\end{equation}
\end{Definition}

The computation of the high opacity coefficient for a square can be found in \cite[Example 16.2]{braides98a}. The high opacity coefficients for other sets have been considered in \cite{concordel97a,c03}. It would be an interesting problem to determine a class of $\Omega_g$ where \eqref{3:hcc} does not hold. It is reasonable to conjecture that if $\partial \Omega_g$ were to have a cusp then \eqref{3:hcc} would fail to hold.

The following lemma provides a sufficient condition on $\Omega_g$ to ensure the existence of a high opacity coefficient.

\begin{Lemma}\label{3:existhcc}
Let $\Omega_g$ be an admissible set. If $\partial \Omega_g$ is connected and differentiable then there exists a high opacity coefficient $\lambda$ for $\Omega_g$.
\end{Lemma}

\begin{proof}
By assumption $\partial \Omega_g$ is a differentiable manifold, therefore it can be equipped with a Riemannian metric $d_{\Omega_g}$ \cite[Theorem 1.4.1]{jost05a}. The manifold has the distance function $d_{\partial \Omega_g}$. The distance function is bi-Lipschitz equivalent to the Euclidean distance \cite[Corollary 1.4.1]{jost05a} since $\partial \Omega_g$ is compact. That is to say that there exist positive constants $c$ and $C$ such that,
\begin{equation}\label{3:456}
c\|x-y\| \leq d_{\partial \Omega_g}(x,y) \leq C \|x-y\|,
\end{equation}
for $x,y \in \partial \Omega_g$. It follows trivially that
\begin{equation}
d_{\Omega_g}(x,y;\beta) \geq \inf_{u \in \mathscr A(x,y)} \int_0^1 \beta \|u'(\tau) \| \, \text{d}\tau = \beta \|x - y \|.\label{3:lab1}
\end{equation}
The inequality in \eqref{3:lab1} follows from the fact that $d_{\Omega_g}(x,y;\beta)$ is an infimum over a smaller space than $\mathscr A(x,y)$. Choosing $\beta > \lambda := C$ then combining \eqref{3:456} and \eqref{3:lab1} proves \eqref{3:hcc}.
\end{proof}

The above lemma is sufficient for a broad class of problems, including those that have physical applications in optics and dynamics. For example, in the dynamical context a discontinuity in the potential, such as that considered here, serves as a microscopic positional constraint. Boundaries of $\Omega_g$ of less regularity are still of mathematical interest but are not studied here. 

The following lemma states that if the variation of $a$ is sufficiently large, then minimising curves do not enter the high opacity regions. It will be exactly the high opacity coefficient which provides the measure of what sufficiently large means in this context. 

\begin{Lemma}\label{3:purpose}
Let $\Omega_g$ be an admissible set with a high opacity coefficient $\lambda$. Define the function
\begin{equation}
a(x) := 
\begin{cases}
\beta & \text{ if } x \in \Omega_g + \mathbb Z^d,\\
1 & \text{otherwise,}
\end{cases}
\end{equation}
for $\beta > \lambda$, and the Riemannian length functional
\begin{equation}
L(u) := \int_0^1 a(u(\tau))\|u'(\tau)\| \, \text{d}\tau.
\end{equation}
For any $x,y \in \Omega_w$, let $u$ be a geodesic joining $x$ to $y$. It then follows that $u(\tau) \in \Omega_w$ for all $\tau \in (0,1)$.
\end{Lemma}

\begin{proof}
Suppose that $\text{graph}(u) \cap \left( \Omega_g + \mathbb Z^d \right) \neq \emptyset$, then there exists $\mathbf x \in \mathbb Z^d$ such that $G(\mathbf x) := \text{graph}(u) \cap \left( \Omega_g + \mathbf x \right) \neq \emptyset$.  Set $TG(\mathbf x) := \{ \tau \in (0,1) \colon u(\tau) \in G(\mathbf x) \}$, $s = \inf TG(\mathbf x)$ and $t = \sup TG(\mathbf x)$.

As $\Omega_g$ is open, it follows that $s<t$, to see this, observe by assumption that $TG(\mathbf x) \neq \emptyset$ and therefore there exists $t' \in TG(\mathbf x)$. By definition $u(t') \in G(\mathbf x) \subset \Omega_g + \mathbf x$, hence there exists $\rho > 0$ such that $B_{\rho}(u(t')) \subset \Omega_g + \mathbf x$, since $\Omega_g + \mathbf x$ is open. By the continuity of $u$ there exists $\delta > 0$ such that $\sigma \in (t' - \delta, t' + \delta)$ implies that $u(\sigma) \in \Omega_g + \mathbf x$. By construction $ s \leq t'-\delta < t' + \delta \leq t$, hence $s<t$. By the continuity of $u$ it holds that $u(s),u(t) \in \Omega_w$.

Further, it holds that, since $u(\sigma) \in \Omega_g + \mathbf x$ for all $\sigma \in (s,t)$,
\begin{equation*}
\int_s^t a(u(\tau))\|u'(\tau)\| \, \text{d}\tau = \int_s^t \beta \|u'(\tau)\| \, \text{d}\tau \geq d_{\Omega_g}(u(s),u(t);\beta).
\end{equation*}
As $\beta > \lambda$, using the invariance of length under reparameterisations, it holds that,
\begin{equation}\label{3:strmin}
\int_s^t a(u(\tau))\|u'(\tau)\| \, \text{d}\tau = \int_s^t \beta \|u'(\tau)\| \, \text{d}\tau > d_{\partial \Omega_g}(u(s),u(t);\beta).
\end{equation}
Since $\partial \Omega_g$ is compact, it follows by the Hopf-Rinow Theorem \cite[Theorem 2.5.28]{burago01a} the infimum in the definition of $d_{\partial \Omega_g}$ is obtained by some $\tilde u \in \mathscr A(u(s),u(t))$. Define the function $v$ by
\begin{equation*}
v(\tau) := \begin{cases}
\tilde u( \tau) & \text{ if } \tau \in (s,t)\\
u(\tau) & \text{ otherwise,}
\end{cases}
\end{equation*}
and using \eqref{3:strmin} gives that $L(u) > L(v)$. Hence $u$ is not a geodesic. 
\end{proof}

\section{Statement of the mathematical problem}
Let $\Omega_g$ satisfy the hypotheses of Lemma \ref{3:purpose} and let the high opacity coefficient be denoted by $\lambda$. Define the metric coefficient $a_{p,\ep} \colon \mathbb R^d \rightarrow \mathbb{R}$ by 
\begin{equation}
a_{p,\ep}(x) := 
\begin{cases}
\beta \ep^{-p} & \text{ if } x \in \Omega_g + \mathbb Z^d,\\
1 & \text{otherwise,}
\end{cases}
\end{equation}
for $\beta > \lambda$, $p \in (0,\infty)$ and $\beta/\lambda > \ep^p > 0$. The requirement that $\beta/\lambda > \ep^p$ ensures that Lemma \ref{3:purpose} holds for all $\ep$. Extend this definition to the case when $p = \infty$ by setting,
\begin{equation}
a_{\infty,\ep}(x) := 
\begin{cases}
+ \infty & \text{ if } x \in \Omega_g + \mathbb Z^d,\\
1 & \text{otherwise,}
\end{cases}
\end{equation}
for $\ep > 0$. For the remainder of this paper we will assume $\beta > \lambda$ is fixed.
Set
\begin{equation}\label{len1}
F_{p,\ep}(u) := \int_0^1 a_{p,\ep} \left( \frac{u(\tau)}{\ep} \right)\| u'(\tau) \| \; \, \text{d}\tau, \quad u \in W^{1,\infty}\left((0,1)\right).
\end{equation}
The aim of this paper is to study two $\G$-convergence problems for $F_{p,\ep}$. In the first problem we examine the $\G$-convergence of $F_{p,\ep}$ on $W^{1,\infty}\left((0,1)\right)$. The second problem is to study the $\G$-convergence of $F_{p,\ep}$ on $\mathscr A(x,y)$. The $\G$-convergence is with respect to the strong $L^{\infty}\left((0,1)\right)$ topology. The main tool behind our arguments is to compare the $\G$-convergence of $F_{p,\ep}$ with the $\G$-convergence of the functionals
\begin{equation}\label{3:llamas}
F_{\ep}(u) := \int_0^1 a\left( \frac{u(\tau)}{\ep} \right) \|u'(\tau) \| \, \text{d}\tau,
\end{equation}
where the metric coefficient is given by 
\begin{equation}\label{metcf}
a(x) := 
\begin{cases}
\beta & \text{ if } x \in \Omega_g + \mathbb Z^d,\\
1 & \text{otherwise.}
\end{cases}
\end{equation}
The computation of the $\G$-limit in this case relies on the fact that the functionals $F_{\ep}$ are uniformly bounded in $\ep$. By \cite[Theorem 15.4]{braides98a} functionals \eqref{3:llamas} $\G$-converge to a functional of the form
\begin{equation}\label{f0}
F_0(u) := \int_0^1 \psi(u'(\tau)) \, \text{d}\tau,
\end{equation}
where the convex function $\psi$ is given by the \emph{asymptotic homogenisation formula} 
\begin{equation}\label{old:asymp}
\psi(\xi) = \lim_{\ep \rightarrow 0} \inf_{u \in \mathscr A (0,\xi)} \int_0^1a\left(\frac{u(\tau)}{\ep}\right)\|u'(\tau)\| \, \text{d}\tau.
\end{equation} 
The function $\psi$ is also 1-homogeneous as demonstrated in \cite{braides02b}. By the fundamental theorem of $\G$-convergence \cite[Theorem 7.2]{braides98a} it follows that
\begin{equation*}
\lim_{\ep \rightarrow 0} d_{\ep}(\xi_1,\xi_2) = \min_{u \in \mathscr A(\xi_1,\xi_2)} \int_0^1 \psi(u'(\tau)) \, \text{d}\tau.
\end{equation*}
Furthermore, it holds that
\begin{equation*}
 \min_{u \in \mathscr A(\xi_1,\xi_2)} \int_0^1 \psi(u'(\tau)) \, \text{d}\tau \leq  \int_0^1 \psi(\xi_2-\xi_1) \, \text{d}\tau = \psi(\xi_2-\xi_1).
\end{equation*}
Note that the minimum exists by \cite[Theorem 7.2]{braides98a}. In addition it holds that for any $u \in \mathscr A(\xi_1,\xi_2)$
\begin{equation}\label{final1}
\psi(\xi_2-\xi_1) = \psi \left(\int_0^1 u'(\tau) \, \text{d}\tau \right) \leq \int_0^1 \psi(u'(\tau)) \, \text{d}\tau
\end{equation}
by Jensen's inequality. Then taking the minimum over all $u \in \mathscr A(\xi_1,\xi_2)$ in \eqref{final1} we get
\begin{equation}\label{1:norm}
\lim_{\ep \rightarrow 0} d_{\ep}(\xi_1,\xi_2) = \psi(\xi_2-\xi_1).
\end{equation}
The distance function induced by $F_{p,\ep}$ is given by
\begin{equation*}
d_{p,\ep}(\xi_1, \xi_2) =  \min_{u \in \mathscr A (\xi_1,\xi_2)} F_{p,\ep}(u),
\end{equation*}
similarly the distance function induced by $F_{\ep}$ is given by
\begin{equation*}
d_{\ep}(\xi_1, \xi_2) =  \min_{u \in \mathscr A (\xi_1,\xi_2)}  F_{\ep}(u).
\end{equation*}
The proof that $d_{p,\ep}$ is a distance function can be found in \cite[Lemma 1.4.1]{jost05a}. The minimum exists by the Hopf-Rinow Theorem \cite[Theorem 2.5.28]{burago01a}. It also follows that
\begin{equation}\label{3:7:est}
\|\xi_1 - \xi_2\| \leq d_{p,\ep}(\xi_1, \xi_2) \leq \frac{\beta}{\ep^p}\|\xi_1 - \xi_2\|.
\end{equation}

\section{$\Gamma$-convergence of the unconstrained metrics}
In this section we address the first problem described before. That is we determine that the sequence $F_{p,\ep}$ $\G$-converges on $W^{1,\infty}\left((0,1)\right)$ with respect to the strong $L^{\infty}\left((0,1)\right)$ topology for $p \in (0,\infty]$. Let $\Omega_g$ satisfy the hypotheses of Lemma \ref{3:purpose} and let the high opacity coefficient be denoted by $\lambda$. First we need a technical lemma which states that given $\ep > 0$ and $u \in W^{1,\infty}\left((0,1)\right)$ there exists a curve no further than $\sqrt{d}\ep$ away from $\text{graph}(u)$ that does not enter the high opacity regions of the Riemannian length density, that is to say that $\text{graph}(u) \cap  \ep \left( \Omega_g + \mathbb{Z}^d \right) = \emptyset$. We remark that the results of this subsection generalise the examples of \cite{buttazzo01a}, in the sense that the assumptions on $\Omega_g$ are relaxed.

\begin{Lemma}\label{3:90}
Let $u \in W^{1,\infty}\left((0,1)\right)$, then for each $\ep > 0$ there exists $u^w_{\ep} \in W^{1,\infty}\left((0,1)\right)$ such that $\text{graph}(u) \subset \ep \Omega_w$ and $\|u - u^w_{\ep} \|_{\infty} \leq \sqrt{d} \ep$.
\end{Lemma}

\begin{proof}
Fix $u \in W^{1,\infty}\left((0,1)\right)$ and $\ep > 0$. Since $\|u\|_{\infty} < \infty$, it follows that there exist $\mathbf x_1, ... , \mathbf x_n \in \mathbb Z^d$ such that $\text{graph}(u) \subset \cup_{i = 1}^n \ep ( [0,1]^d + \mathbf x_i)$. Fix $i \in \{ 1, ..., n \}$ and define $G_i := \text{graph}(u) \cap \ep \left( \Omega_g + \mathbf x_i \right)$.  Let $G_i^j$ be a connected component of $G_i$; there exists finitely many such connected components since $u$ is Lipschitz. Now fix $j$. Set $TG_i^j := \{ \tau \in (0,1) \colon  u(\tau) \in G_i^j \}$, $s_i^j = \inf TG_i^j$ and $t_i^j = \sup TG_i^j$. Choose $i$ such that $TG_i^j \ne \emptyset$; if $TG_i^j = \emptyset$ for all $i,j$ then set $u^w_{\ep} = u$ and we are done. Applying the argument of Lemma \ref{3:purpose}, $s_i^j < t_i^j$ and $u(s_i^j), u(t_i^j) \in \partial (\ep ( \Omega_g + \mathbf x_i))$. Since $\partial\Omega_g$ is path connected there exists a Lipschitz curve joining $u(s_i^j)$ to $u(t_i^j)$ in $ \partial (\ep ( \Omega_g + \mathbf x_i))$ denoted as $w_i^j$.  The fact that $w_i^j$ is Lipschitz continuous follows from the smoothness of $\partial \Omega_g$ \cite[Chapters 1 and 8]{jost05a}. Then set
\begin{equation*}
u^w_{\ep}(\tau) := \begin{cases}
w_i^j(\tau) & \text{ if } \tau \in (s_i^j,t_i^j)\\
u(\tau) & \text{ otherwise.}
\end{cases}
\end{equation*}
It is clear from the construction that $\text{graph}(u) \subset \ep \Omega_w$. Note that, since $\Omega_g$ is assumed to be open, $\Omega_w$ is closed. It remains to check that $\|u - u^w_{\ep} \|_{\infty} \leq \sqrt{d} \ep$. Fix, $\tau \in (s_i, t_i)$ for some $i \in \{ 1 , ... , n\}$, then
\begin{align*}
\|u^w_{\ep} - u\|_{\infty} & = \|w_i - u \|_{\infty} \leq \ep \text{diam}(\text{cl}(\Omega_g + \mathbf x_i)),
\end{align*}
since $u(\tau), w_i(\tau) \in \text{cl}(\ep(\Omega_g + \mathbf x_i))$ for all $\tau \in (s_i^j, t_i^j)$. It is immediate that $\text{diam}(\text{cl}(\Omega_g + \mathbf x_i)) \leq \sqrt{d}$ and therefore taking the supremum over all $\tau$ gives the required estimate. The fact that $u^w_{\ep} \in W^{1,\infty}\left((0,1)\right)$ follows from the regularity of $u$, $w_i$ and $\partial \Omega_g$.
\end{proof}

The following lemma shows that for the $\G$-convergence of $F_{\ep}$ it is possible to choose a recovery sequence that never enters the higher opacity region. 

\begin{Lemma}\label{3:tec:1}
For each $u \in W^{1,\infty}\left((0,1)\right)$ and each sequence $(\ep_k)_{k =1}^{\infty}$ converging to $0$, there exists a sequence $(u_{\ep_k})_{\ep_k > 0} \subset W^{1,\infty}\left((0,1)\right)$, converging in $L^{\infty}\left((0,1)\right)$ to $u$, such that
\begin{enumerate}
\item $\lim_{k \rightarrow \infty} F_{\ep_k}(u_{\ep_k}) = F_0(u)$, and,
\item $\text{graph}(u_{\ep_k}) \subset \ep_k \Omega_w$ for all $k$
\end{enumerate}
where $F_0$ is given by \eqref{f0} and \eqref{old:asymp}.
\end{Lemma}

\begin{proof}
Let $K \subset \subset \mathbb R^d$ be such that $\text{graph}(u) \subset \text{int}(K)$. By \cite[Theorem 3.1]{buttazzo01a} the metrics induced by $F_{\ep_k}$, denoted $d_{\ep_k}$, converge locally uniformly to the metric induced by the norm $\psi$ as $k \rightarrow \infty$. Therefore, it is possible to choose $(M_k)_{k=1}^{\infty} \subset \mathbb N$ converging to infinity such that
\begin{equation*}
\lim_{k \rightarrow \infty} M_k \sup_{\xi_1, \xi_2 \in K} \left| d_{\ep_k}(\xi_1,\xi_2) - \psi(\xi_2 - \xi_1) \right| = 0.
\end{equation*}
Let $\pi_{M_k} = \{ \tau_0 ,... ,\tau_{M_k}\}$ be a partition of $[0,1]$ such that $|\tau_j - \tau_{j+1}| = 1/M_k$ for $j = 1, ... ,M_k$. Define the function $u_{\ep_k}$ by
\begin{multline}\label{3:1234}
u_{\ep_k} = u^w_{\ep_k} +\\ \text{argmin}_{w \in W^{1,\infty}_0\left((0,1)\right)} \int_{\tau_{i-1}}^{\tau_i} a\left(\frac{u(\tau)+w(\tau)}{\ep_k}\right)\|u'(\tau) + w'(\tau) \| \, \text{d} \tau
\end{multline}
on $[{\tau_{i-1}},{\tau_i}]$, where $u^w_{\ep_k}$ is given by Lemma \ref{3:90}. The minimiser of \label{3:1234} exists by the Hopf-Rinow Theorem \cite[Theorem 2.5.28]{burago01a}. Fix $k$ and $t \in [0,1]$ and suppose that $\tau \in [\tau_{i-1},\tau_i]$. Then
\begin{align}
\left\| u_{\ep_k}(\tau) - u^w_{\ep_k}(\tau) \right\| & \leq \left\| u_{\ep_k}(\tau) - u_{\ep_k}(\tau_{i-1}) \right\| + \left\|u_{\ep_k}(\tau_{i-1}) - u^w_{\ep_k}(\tau) \right\|, \nonumber\\
& = \left\| u_{\ep_k}(\tau) - u_{\ep_k}(\tau_{i-1}) \right\| + \left\|u_{\ep_k}^w(\tau_{i-1}) - u^w_{\ep_k}(\tau) \right\|, \label{3:13:13}
\end{align}
using the fact that $u^w_{\ep_k} = u_{\ep_k}$ on $\pi_{M_k}$. By the bound on $a$ \eqref{metcf} it holds that
\begin{multline}
\left\| u_{\ep_k}(\tau) - u_{\ep_k}(\tau_{i-1}) \right\| \leq d_{\ep_k}( u_{\ep_k}(\tau),u_{\ep_k}(\tau_{i-1}) ) \\
 \leq d_{\ep_k}( u_{\ep_k}(\tau_i),u_{\ep_k}(\tau_{i-1}) )
 \leq \beta \left\| u_{\ep_k}(\tau_i) - u_{\ep_k}(\tau_{i-1}) \right\|.\label{3:12:12}
\end{multline}
Since $u^w_{\ep_k} \rightarrow u$ uniformly by Lemma \ref{3:90}, it follows that the sequence $u^w_{\ep_k}$ is equicontinuous by the converse of the Arzel\'a-Ascoli Theorem. Therefore, fix $\eta > 0$ then there exists $\delta > 0$ such that $|x-y| < \delta$ implies that $\|u^w_{\ep_k}(x)-u^w_{\ep_k}(y)\| < \eta$ for all $k$. Consequently there exists $K \in \mathbb N$ such that $k \geq K$ implies $|\pi_{M_k}| < \delta$, therefore for $k \geq K$, combining \eqref{3:13:13} and \eqref{3:12:12},
\begin{equation*}
\left\| u_{\ep_k}(\tau) - u_{\ep_k}^w(\tau) \right\| \leq \beta \eta + \eta.
\end{equation*}
Consequently, since $\eta$ was arbitrary, $\|u_{\ep_k} - u_{\ep_k}^w\|_{\infty} \rightarrow 0$ as $k \rightarrow \infty$, and since $u^w_{\ep_k}$ converges to $u$ in $L^{\infty}\left((0,1)\right)$ it holds that $u_{\ep_k} \rightarrow u$ in $L^{\infty}\left((0,1)\right)$. It remains to show that $u_{\ep_k}$ has the desired properties. Observe that
\begin{align}
\int_0^1\psi(u'(\tau)) \, \text{d}\tau & \geq \sum_{i=1}^{M_k} \psi(u(\tau_i)-u(\tau_{i-1})) \nonumber \\
& \geq \sum_{i=1}^{M_k}  d_{\ep}(u_{\ep_k}(\tau_i),u_{\ep_k}(\tau_{i-1}))  \, -\nonumber \\
& \qquad \sum_{i=1}^{M_k} \left| \psi(u(\tau_i)-u(\tau_{i-1})) - d_{\ep}(u_{\ep_k}(\tau_i),u_{\ep_k}(\tau_{i-1})) \right|.\label{3:th1111}
\end{align}
The first inequality in the above holds since the length functional with density $\psi$ gives rise to an induced metric $d(x,y) = \psi(x-y)$, the equality following from \eqref{1:norm}. By construction it is true that
\begin{equation*}
 \sum_{i=1}^{M_k}  d_{\ep}(u_{\ep_k}(\tau_i),u_{\ep_k}(\tau_{i-1})) = F_{\ep}(u_{\ep_k}),
\end{equation*}
furthermore, since there exists $k_0$ such that $\text{graph}(u_{\ep_k}) \subset K$ for all $k \geq k_0$,
\begin{multline*}
\sum_{i=1}^{M_k} \left| \psi(u(\tau_i)-u(\tau_{i-1})) - d_{\ep}(u_{\ep_k}(\tau_i),u_{\ep_k}(\tau_{i-1})) \right| \\ \leq M_k \sup_{\xi_1, \xi_2 \in K} |d_{\ep_k}(\xi_1,\xi_2) - \psi(\xi_2 - \xi_1)|.
\end{multline*}
Hence by \eqref{3:th1111}, the choice of $(M_k)_{k=1}^{\infty}$ and the liminf inequality,
\begin{equation*}
\int_0^1\psi(u'(\tau)) \, \text{d}\tau \geq \limsup_{k \rightarrow \infty} F_{\ep}(u_{\ep_k}) \geq \liminf_{k \rightarrow \infty} F_{\ep}(u_{\ep_k}) \geq \int_0^1\psi(u'(\tau)) \, \text{d}\tau.
\end{equation*}
The last thing we need to show is that $\text{graph}(u_{\ep_k}) \subset \ep_k \Omega_w$ for all $k$. By the construction, as given in \eqref{3:1234}, $u_{\ep_k}$ is constructed as a piecewise geodesic curve joining points along $u^w_{\ep_k}$. By Lemma \ref{3:purpose} these geodesic pieces do not enter $\Omega_g + \mathbb Z^d$, hence the result holds.
\end{proof}

\begin{Theorem}\label{thm:lenconv}
Let $\Omega_g$ satisfy the hypotheses of Lemma \ref{3:purpose} and let the high opacity coefficient be denoted by $\lambda$. Let $\beta > \lambda$, then
\begin{equation*}
\G(L^{\infty}\left((0,1)\right))-\lim F_{p,\ep}(u) = F_0(u), \; \forall u \in W^{1,\infty}\left((0,1)\right), \forall p \in (0,\infty]
\end{equation*}
where $F_0$ is defined by equations \eqref{f0} and \eqref{old:asymp}.
\end{Theorem}

\begin{proof}
Fix $p \in (0,\infty]$ and let $(\ep_k)_{k = 1}^{\infty}$ be a sequence converging to 0. Without relabelling, pass to the subsequence where $\ep_k^p < \beta/\lambda$ for all $k$. Let $u_{\ep_k} \rightarrow u$ in $L^{\infty}\left((0,1)\right)$. Since $a_{p,\ep_k} \geq a$, by our choice of $\ep$, it follows that
\begin{equation*}
\liminf_{k \rightarrow \infty} F_{p,\ep_k}(u_{\ep_k}) \geq \liminf_{k \rightarrow \infty} F_{\ep_k} (u_{\ep_k}) \geq F_0(u),
\end{equation*}
the second inequality being the liminf inequality for the $\G$-convergence of $F_{\ep_k}$. Applying Lemma \ref{3:tec:1} we obtain a sequence $u_{\ep_k}$ converging to $u$ in $L^{\infty}\left((0,1)\right)$ where $\lim_{k \rightarrow \infty} F_{\ep_k}(u_{\ep_k}) = F_0(u)$ and  $\text{graph}(u_{\ep_k}) \subset \ep_k \Omega_w$ for all $k$. Since $a_{p,\ep_k} = a$ on $\Omega_w$ it follows that
\begin{equation*}
\lim_{k \rightarrow \infty} F_{p,\ep_k}(u_{\ep_k}) = \lim_{k \rightarrow \infty} F_{\ep_k}(u_{\ep_k}).
\end{equation*}
Hence the sequence $F_{p,\ep_k}$ $\G$-converges by the Urysohn property of $\G$-convergence \cite[Proposition 7.11]{braides98a}.
\end{proof}

While the proof of Theorem \ref{thm:lenconv} is nontrivial, as the sequence of functionals are not bounded uniformly, the convergence of minimisers, and hence the motivation for computing the $\G$-limit, can be determined using a simpler argument. Observe, by the growth conditions that the minimiser of $F_{p,\ep}$ is the zero function, which is unique and this converges to the minimiser of $F_0$. The fact that the $0$ function is the unique minimiser of $F_0$ on $W^{1,\infty}\left((0,1)\right)$ follows from the growth estimates on $\psi$ \cite{amar98a}. It proves to be more interesting to understand the $\G$-convergence of $F_{p,\ep}$ on a smaller space, typically the space of Lipschitz curves joining two fixed points. The result of such an analysis may then be applied to problems in nonlinear optics and dynamics. In this setting, the minimal curves are nontrivial, and to understand their effective behaviour proves to be more challenging in the context of the unbounded length functionals we consider here. This is the subject of the next section.

\section{$\Gamma$-convergence for the boundary value problem}
As before, let $\Omega_g$ satisfy the hypotheses of Lemma \ref{3:purpose} and let the high opacity coefficient be denoted by $\lambda$. In the previous subsection it was shown that the unbounded length functionals $F_{p,\ep}$ $\G$-converge on $W^{1,\infty}\left((0,1)\right)$, for all $p \in (0,\infty]$. From this no additional information may be derived to approximate geodesics joining points. Typically, once the $\G$-limit is calculated for the unconstrained problem, the $\G$-limit with boundary conditions can be calculated as an extension of the original argument \cite[Proposition 11.7]{braides98a}. The extension argument relies on what is known as the \emph{fundamental estimate}, which allows one to join the functions in the argument of the functional, introducing only a small error term.

In this subsection we study the effective description of geodesics joining two points for the length functional \eqref{len1}. The main, and indeed surprising, result here is that the existence of the $\G$-limit for $F_{p,\ep}$ depends on the value of $p$. To prove this fact we show that the sequence $\G$-converges on the space of curves joining any two points if and only if the induced metric converges locally uniformly; this result is that of \cite{buttazzo01a} but in the context of non-uniformly bounded two phase metrics. Here we show that the induced metrics always converge locally uniformly, should the opacity coefficient grow sufficiently slowly as $\ep \rightarrow 0$.

\subsection{The induced metrics converge locally uniformly to a norm for $p < 1$}

In this subsection we consider the case of slowly increasing opacity with $\ep$. Here we prove that the induced metrics $d_{p,\ep}$ converge locally uniformly to a norm. The following technical lemma allows one to find points in $\ep\Omega_w$ that are within a uniform distance of any point in $\mathbb{R}^d$.

\begin{Lemma}\label{3:lem:exseq}
For all $\xi \in \mathbb R^d, \ep > 0$ there exists $\xi_{\ep} \in \ep \Omega_w$ be such that $\|\xi_{\ep} - \xi \| \leq \sqrt{d}\ep$.
\end{Lemma}

\begin{proof}
Trivial. 
\end{proof}

In the case when $p < 1$ the following lemma describes the pointwise convergence of $d_{p,\ep}$ in terms of the convergence of $d_{p,\ep}$ on $\ep\Omega_w$. The latter convergence is easier to prove as there is more information available about geodesics joining points in $\ep \Omega_w$, c.f. Lemma \ref{3:purpose}.

\begin{Lemma}\label{3:lem:coex}
Let $p < 1$, $\xi_1,\xi_2 \in \mathbb R^d$, $\ep >0$ and $\xi_{1,\ep},\xi_{2,\ep}  \in \ep \Omega_w$ be such that $\|\xi_{1,\ep} - \xi_1 \| \leq \sqrt{d}\ep$, $\|\xi_{2,\ep} - \xi_2 \| \leq \sqrt{d}\ep$. The limit 
\begin{equation}\label{3:AA}
\lim_{\ep \rightarrow 0} d_{p,\ep}(\xi_1, \xi_2) = \lim_{\ep \rightarrow 0} \min_{u \in \mathcal A(\xi_1, \xi_2)} F_{p,\ep}(u) 
\end{equation}
exists if and only if the limit 
\begin{equation}\label{3:BB}
\lim_{\ep \rightarrow 0} d_{p,\ep}(\xi_{1,\ep}, \xi_{2,\ep}) = \lim_{\ep \rightarrow 0} \min_{u \in \mathcal A(\xi_{1,\ep},\xi_{2,\ep})} F_{p,\ep}(u) 
\end{equation}
exists. 
\end{Lemma}

\begin{proof}
Fix $\ep^p \in (0,\beta/\lambda)$. Then, using the triangle inequality and \eqref{3:7:est},
\begin{align*}
 d_{p,\ep}(\xi_1, \xi_2) &\leq d_{p,\ep}(\xi_{1}, \xi_{1,\ep}) + d_{p,\ep}(\xi_{1,\ep}, \xi_{2,\ep}) + d_{p,\ep}(\xi_{2,\ep}, \xi_2)\\
& \leq \frac{\beta}{\ep^p}\|\xi_{1} - \xi_{1,\ep}\| + d_{p,\ep}(\xi_{1,\ep}, \xi_{2,\ep}) +  \frac{\beta}{\ep^p}\|\xi_{2,\ep} - \xi_2\| \\
& \leq d_{p,\ep}(\xi_{1,\ep}, \xi_{2,\ep}) +  2\beta \sqrt{d} \ep^{1-p}.
\end{align*}
Furthermore,
\begin{align*}
 d_{p,\ep}(\xi_1, \xi_2) &\geq d_{p,\ep}(\xi_{1,\ep}, \xi_{2,\ep}) - d_{p,\ep}(\xi_{1,\ep}, \xi_1) - d_{p,\ep}(\xi_{2}, \xi_{2,\ep})\\
& \geq  d_{p,\ep}(\xi_{1,\ep}, \xi_{2,\ep}) -  \frac{\beta}{\ep^p}\|\xi_{1,\ep} - \xi_1\|  -  \frac{\beta}{\ep^p}\|\xi_{2} - \xi_{2,\ep}\|\\
& \geq  d_{p,\ep}(\xi_{1,\ep}, \xi_{2,\ep}) -  2\beta \sqrt{d} \ep^{1-p}.
\end{align*}
Therefore
\begin{equation}\label{3:CC}
\left|d_{p,\ep}(\xi_1, \xi_2) - d_{p,\ep}(\xi_{1,\ep}, \xi_{2,\ep}) \right| \leq 2\beta \sqrt{d} \ep^{1-p}.
\end{equation}
If the limit \eqref{3:BB} exists then taking the limit as $\ep \rightarrow 0$ gives that, by \eqref{3:CC}, the limit \eqref{3:AA} exists. The converse statement is proved in an identical fashion. 
\end{proof}

The next lemma provides an 'unfolding' mechanism for a metric which replaces the minimisation of an $\ep$-dependant length functional with the minimisation of a single scale metric joining $\ep$-dependant end points. This observation ensures that Lemma \ref{3:purpose} can be applied to the unfolded metric. 

\begin{Lemma}\label{3:lem:cov}
Let $p < 1$, $\xi_1,\xi_2 \in \mathbb R^d$, $\ep >0$ and $\xi_{1,\ep},\xi_{2,\ep}  \in \ep \Omega_w$ be such that $\|\xi_{1,\ep} - \xi_1 \| \leq \sqrt{d}\ep$, $\|\xi_{2,\ep} - \xi_2 \| \leq \sqrt{d}\ep$. Then
\begin{equation*}
\lim_{\ep \rightarrow 0} \min_{u \in \mathscr A(\xi_{1,\ep},\xi_{2,\ep})} F_{p,\ep}(u) = \\
\lim_{\ep \rightarrow 0} \min_{u \in \mathscr A(\xi_{1,\ep}/\ep,\xi_{2,\ep}/\ep)} \ep \int_0^{1} a_{p,\ep} (u(\tau))\|u'(\tau)\| \, \text{d}\tau ,
\end{equation*}
provided either limit exists.
\end{Lemma}

\begin{proof}
Fix $\ep > 0$ and $u \in \mathscr A(\xi_{1,\ep},\xi_{2,\ep})$, then
\begin{align*}
F_{p,\ep}(u) &= \int_0^1 a_{p,\ep}\left( \frac{u(\tau)}{\ep} \right) \|u'(\tau)\| \, \text{d}\tau = \ep \int_0^1 a_{p,\ep}\left( w(\tau) \right) \|w'(\tau)\| \, \text{d}\tau,
\end{align*}
where $w = u/\ep$. Clearly $w \in \mathscr A(\xi_{1,\ep}/\ep,\xi_{2,\ep}/\ep)$. Since the correspondence between $u$ and $w$ is one-to-one the result follows by taking the minimum over $u$, or equivalently $w$, and then passing to the limit.
\end{proof}

The following lemma connects the homogenisation of $F_{p,\ep}$ with that of $F_{\ep}$. The proof demonstrates that the functionals $F_{\ep}$ and $F_{p,\ep}$ give equivalent measures of length for geodesics. For notational convenience we define for $x,y \in \mathbb{R}^d$ the space $\mathscr A_w(x,y) := \{ u \in W^{1,\infty}\left((0,1)\right) : u(0) = x, u(1) = y, u(\tau) \in \Omega_w \, \forall \tau \}$.

\begin{Lemma}\label{3:le1}
Let $p < 1$, $\xi_1,\xi_2 \in \mathbb R^d$, $\ep >0$ and $\xi_{1,\ep},\xi_{2,\ep}  \in \ep \Omega_w$ be such that $\|\xi_{1,\ep} - \xi_1 \| \leq \sqrt{d}\ep$, $\|\xi_{2,\ep} - \xi_2 \| \leq \sqrt{d}\ep$. For $\beta/\lambda > \ep > 0$ it follows that
\begin{multline*}
\lim_{\ep \rightarrow 0} \min_{u \in \mathscr A(\xi_{1,\ep}/\ep,\xi_{2,\ep}/\ep)} \ep \int_0^{1} a_{p,\ep} (u(\tau))\|u'(\tau)\| \, \text{d}\tau = \\
\lim_{\ep \rightarrow 0} \min_{u \in \mathscr A(\xi_{1,\ep}/\ep,\xi_{2,\ep}/\ep)} \ep \int_0^{1} a (u(\tau))\|u'(\tau)\| \, \text{d}\tau .
\end{multline*}
\end{Lemma}

\begin{proof}
Fix  $\ep^p \in (0,\beta/\lambda)$, then
\begin{multline}\label{3:65:1}
\min_{u \in \mathscr A(\xi_{1,\ep}/\ep,\xi_{2,\ep}/\ep)} \ep \int_0^{1} a_{p,\ep} (u(\tau))\|u'(\tau)\| \, \text{d}\tau  = \\
\min_{u \in\mathscr A_w(\xi_{1,\ep}/\ep,\xi_{2,\ep}/\ep)} \ep \int_0^{1} a_{p,\ep} (u(\tau))\|u'(\tau)\| \, \text{d}\tau
\end{multline}
by Lemma \ref{3:purpose} since $\xi_{1,\ep}/\ep,\xi_{2,\ep}/\ep \in \Omega_w$. As $a_{p,\ep} = a$ on $\Omega_w$
\begin{multline}\label{3:65:2}
\min_{u \in \mathscr A_w(\xi_{1,\ep}/\ep,\xi_{2,\ep}/\ep)} \ep \int_0^{1} a_{p,\ep} (u(\tau))\|u'(\tau)\| \, \text{d}\tau  = \\
\min_{u \in\mathscr A_w(\xi_{1,\ep}/\ep,\xi_{2,\ep}/\ep)} \ep \int_0^{1} a (u(\tau))\|u'(\tau)\| \, \text{d}\tau.
\end{multline}
Another application of Lemma \ref{3:purpose} gives that
\begin{multline}\label{3:65:3}
\min_{u \in \mathscr A_w(\xi_{1,\ep}/\ep,\xi_{2,\ep}/\ep)} \ep \int_0^{1} a (u(\tau))\|u'(\tau)\| \, \text{d}\tau  = \\
\min_{u \in\mathscr A(\xi_{1,\ep}/\ep,\xi_{2,\ep}/\ep)} \ep \int_0^{1} a (u(\tau))\|u'(\tau)\| \, \text{d}\tau.
\end{multline}
Equating \eqref{3:65:1}, \eqref{3:65:2} and \eqref{3:65:3} and taking the limit as $\ep \rightarrow 0$ gives the result. 
\end{proof}

The following lemma is a statement of the analogues of Lemmas \ref{3:lem:coex} and \ref{3:lem:cov} for the sequence $F_{\ep}$.

\begin{Lemma}\label{3:lem:1s}
Let $p < 1$, $\xi_1,\xi_2 \in \mathbb R^d$, $\ep >0$ and $\xi_{1,\ep},\xi_{2,\ep}  \in \ep \Omega_w$ be such that $\|\xi_{1,\ep} - \xi_1 \| \leq \sqrt{d}\ep$, $\|\xi_{2,\ep} - \xi_2 \| \leq \sqrt{d}\ep$. Then
\begin{equation*}
\lim_{\ep \rightarrow 0} \min_{u \in \mathscr A(\xi_{1,\ep},\xi_{2,\ep})} F_{\ep}(u) =
\lim_{\ep \rightarrow 0} \min_{u \in \mathscr A(\xi_{1,\ep}/\ep,\xi_{2,\ep}/\ep)} \ep \int_0^1 a (u(\tau))\|u'(\tau)\| \, \text{d}\tau .
\end{equation*}
Furthermore, the limit 
\begin{equation}\label{3:AAa}
 \lim_{\ep \rightarrow 0} \min_{u \in \mathcal A(\xi_1, \xi_2)} F_{\ep}(u)
\end{equation}
exists, if and only if, the limit 
\begin{equation}\label{3:BBa}
 \lim_{\ep \rightarrow 0} \min_{u \in \mathcal A(\xi_{1,\ep}, \xi_{2,\ep})} F_{\ep}(u)
\end{equation}
exists. 
\end{Lemma}

\begin{proof}
The first part follows exactly as in Lemma \ref{3:lem:cov}. The second part is demonstrated using an identical proof to that for Lemma \ref{3:lem:coex}.
\end{proof}

The following theorem combines Lemmas \ref{3:lem:coex}, \ref{3:lem:cov} and \ref{3:lem:1s} to establish the pointwise convergence of metrics for $p < 1$.

\begin{Theorem}\label{3:lem:exi}
Let $p <1$. Then the limit 
\begin{equation}\label{3:AABB}
\lim_{\ep \rightarrow 0} d_{p,\ep}(\xi_1,\xi_2) = \lim_{\ep \rightarrow 0} \min_{u \in \mathcal A(\xi_1, \xi_2)} F_{p,\ep}(u) = \psi(\xi_2 - \xi_1)
\end{equation}
exists for all $\xi_1,\xi_2 \in \mathbb R^d$, where $\psi$ is as in \eqref{old:asymp}. Furthermore,
\begin{equation}
\lim_{\ep \rightarrow 0} d_{p,\ep}(\xi_1,\xi_2) = \lim_{\ep \rightarrow 0} \min_{u \in \mathcal A(\xi_1, \xi_2)} F_{\ep}(u).
\end{equation}
\end{Theorem}

\begin{proof}
By \cite[Proposition 3.2]{amar98a} the limit 
\begin{equation*}
\lim_{\ep \rightarrow 0} \min_{u \in \mathcal A(\xi_1, \xi_2)} F_{\ep}(u)
\end{equation*}
exists. Applying Lemma \ref{3:lem:exseq}, for each $\ep > 0$, we obtain the existence of $\xi_{1,\ep},\xi_{2,\ep}  \in \ep \Omega_w$ be such that $\|\xi_{1,\ep} - \xi_1 \| \leq \sqrt{d}\ep$, $\|\xi_{2,\ep} - \xi_2 \| \leq \sqrt{d}\ep$. The result then follows by sequentially applying Lemmas \ref{3:lem:1s}, \ref{3:le1}, \ref{3:lem:coex} and \ref{3:lem:cov}. 
\end{proof}

\subsection{The induced metrics fail to converge pointwise for $p \geq 1$}

In this section we show that for $p \geq 1$ the limit
\begin{equation*}
\lim_{\ep \rightarrow 0} d_{p,\ep}(\xi_1,\xi_2) = \lim_{\ep \rightarrow 0} \min_{u \in \mathcal A(\xi_1, \xi_2)} F_{p,\ep}(u)
\end{equation*}
does not exist, in contrast to the case $p < 1$ it does. Consequently, the sequence $F_{p,\ep}$ fails to $\G$-converge on $\mathcal A(\xi_1,\xi_2)$ for all $\xi_1,\xi_2 \in \mathbb R^d$; since otherwise by the fundamental theorem of $\G$-convergence \cite{braides02a,braides98a} the minimum values would converge. 

The proof relies on a simple geometric assumption. The principle behind the argument is that we may choose, for a specific pair of end points, two different sequences of values for $\ep$ that subsequently give rise to two different limit values. Namely, we may choose a sequence where the endpoints lie in $\ep\Omega_w$ and therefore using Lemma \ref{3:purpose} and subsequent results to find a finite limit. However it is also possible to choose a sequence of values for $\ep$ where one of the end points lies in $\Omega_g + \mathbb{R}^d$, resulting in the divergence of the length functionals in $\ep$.

A sketch of the argument can be found in figure 1. The argument relies on joining $\xi_2 := (1/2,1/2,\cdots) \in \Omega_g$ and $\xi_1 := (0,0,\cdots) \in \Omega_w$ for specific choices of sequences of values for $\ep$. In particular we choose a sequence $\ep_k$ such that $\xi_2 \in \ep_k \left(\Omega_g + \mathbb{Z}^d\right)$ for all $k$. It follows that $d_{p,\ep_k}(\xi_1,\xi_2) = \psi(\xi_2-\xi_1) + O(\ep_k^{1-p})$ as $k \rightarrow \infty$. We also construct a sequence such that $\xi_2 \in \ep_k \Omega_w$ for all $k$. In this case, as the geodesic never enters $\ep_k \left(\Omega_g + \mathbb{Z}^d\right)$, it follows that $d_{p,\ep_k}(\xi_1,\xi_2) = \psi(\xi_2-\xi_1)$ as $k \rightarrow \infty$. Consequently we get different limit values for two sequences of values for $\ep$. Therefore the $\G$-limit ceases to exist.

\begin{figure}[htbp]
\begin{center}
\includegraphics[scale=0.55]{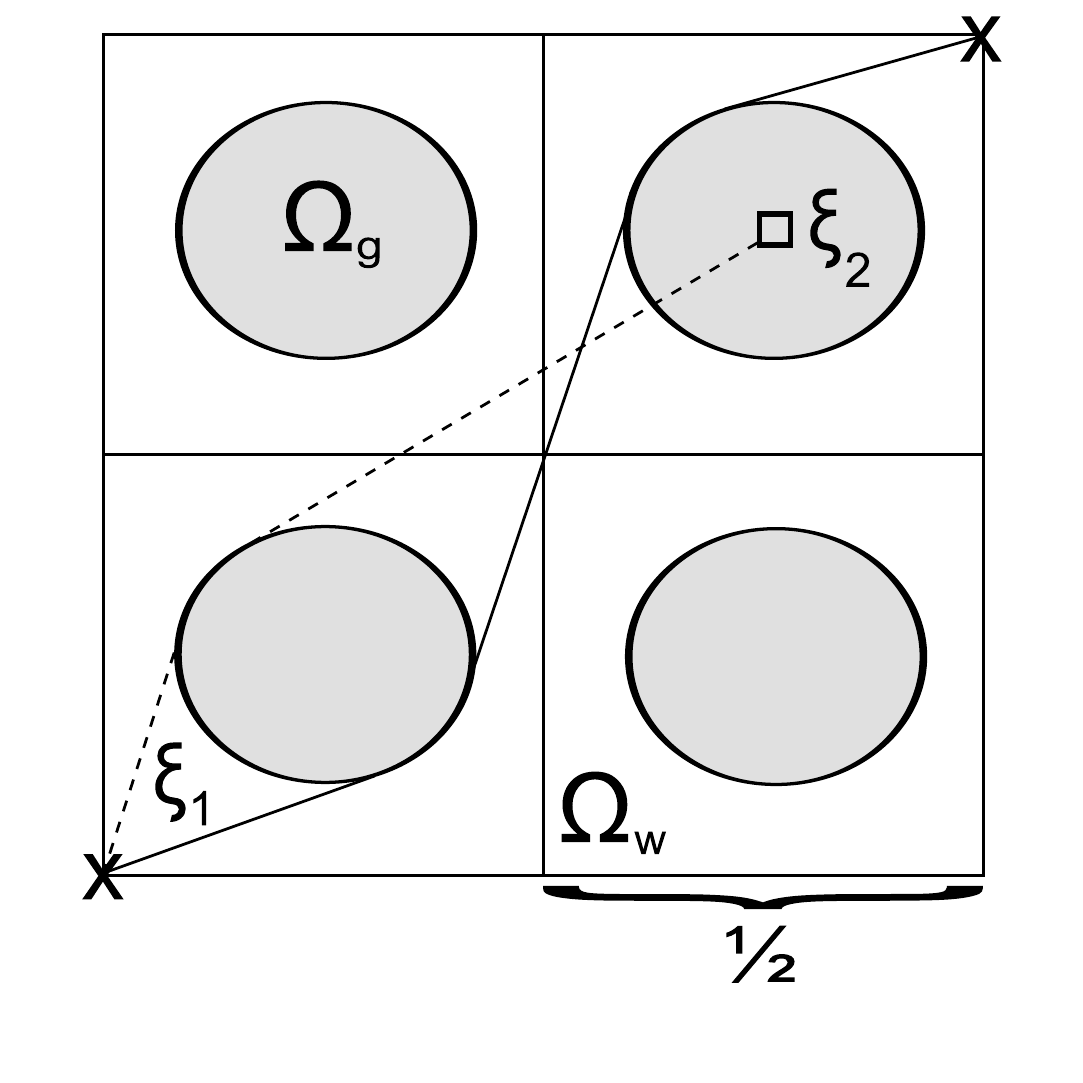}
\caption{An illustration of the proof of Theorem \ref{3:t:65}. The sequence $\tilde \ep_k$ corresponds to a sequence such that the end points of the curve lie in $\ep\Omega_w$, indicated by the solid curve. The sequence $\ep_k$ corresponds to a sequence such that it has an end point in $\ep\left(\Omega_g + \mathbb{Z}^d\right)$, indicated by the dashed curve.}
\label{fig:1}
\end{center}
\end{figure}

\begin{Theorem}\label{3:t:65}
Suppose that $\xi_2 := (1/2,1/2,\cdots) \in \Omega_g$ and that $\xi_1 := (0,0,\cdots) \in \Omega_w$. Then the limit
\begin{equation*}
\lim_{\ep \rightarrow 0} d_{p,\ep}(\xi_1,\xi_2) = \lim_{\ep \rightarrow 0} \min_{u \in \mathscr A(\xi_1,\xi_2)} F_{p,\ep}(u)
\end{equation*}
does not exist.
\end{Theorem}

\begin{proof}
Set 
\begin{equation*}
\tilde \ep_k := \frac{1}{2k}, \quad k \in \mathbb N.
\end{equation*}
If $p < \infty$ we pass to a subsequence such that $\tilde \ep_k^p < \beta/\lambda$. It follows immediately that $\xi_2 \in \tilde \ep_k \mathbb{Z}^d \subset \tilde \ep_k \Omega_w$ for all $k$, since $\xi_1 \in \Omega_w$. Consequently, by Lemma \ref{3:purpose} the solution to the problem \begin{equation*}
\min_{u \in \mathscr A(\xi_1,\xi_2)} F_{p,\tilde \ep_k}(u),
\end{equation*} 
which we denote by $w_{\tilde \ep_k}$, is such that $\text{graph}(w_{\tilde \ep_k}) \subset \tilde \ep_k \Omega_w$. Applying Theorem \ref{3:lem:exi}, with $\xi_{1,\tilde \ep_k} = \xi_1$ and $\xi_{2,\tilde \ep_k} = \xi_2$ for all $\tilde \ep > 0$, it holds that
\begin{equation*}
 \lim_{k \rightarrow \infty} \min_{u \in \mathscr A(\xi_1,\xi_2)}  F_{p,\tilde \ep_k}(u) = \lim_{k \rightarrow \infty} \min_{u \in \mathscr A(\xi_1,\xi_2)} F_{\tilde \ep_k}(u) =  \psi(\xi_2-\xi_1)
\end{equation*} 
We now construct a second sequence converging to a different limit. Set 
\begin{equation}
\ep_k = \frac{1}{2k+1}, \quad k \in \mathbb N.
\end{equation}
If $p < \infty$ we again pass to a subsequence such that $\ep_k^p < \beta/\lambda$. It follows that $\xi_2 \in \ep_k (\Omega_g + \mathbb Z^d)$ for all $k$. Denote the solution to the problem 
\begin{equation*}
\min_{u \in \mathscr A(\xi_1,\xi_2)} F_{p,\ep_k}(u),
\end{equation*} 
by $w_{\ep_k}$. It follows by Lemma \ref{3:purpose} that for each $k$ there exists $\tau_k \in (0,1)$ such that $w_{\ep_k}(\tau) \in \ep_k \Omega_w$ for $\tau \in [0,\tau_k]$ and $w_{\ep_k}(\tau) \in \ep_k(\Omega_g + (k,k,\cdots))$ for $\tau \in (\tau_k,1]$ and $(k,k,\cdots) \in \mathbb{Z}^d$. Since $\Omega_g$ is open there exists a $\rho >0 $ such that $B_{\rho}(\xi_2) \subset \Omega_g$. Hence, for all $k$, $\ep_k (B_{\rho}(\xi_2) + (k,k,\cdots)) \subset \ep_k(\Omega_g + (k,k,\cdots))$. By continuity for each $k$ there exists $\sigma_k \in (\tau_k,1)$ such that $w_{\ep_k}(\sigma_k) \in \partial \left( \ep_k (B_{\rho}(\xi_2) + (k,k,\cdots)) \right)$. Therefore
\begin{align}
F_{p,\ep_k}(w_{\ep_k}) & = \int_0^{\tau_k} a_{p,\ep}\left( \frac{w_{\ep_k}(\tau)}{\ep_k} \right) \| w'_{\ep_k}(\tau) \| \, \text{d}\tau \nonumber \\& \qquad + \int_{\tau_k}^1 a_{p,\ep}\left( \frac{w_{\ep_k}(\tau)}{\ep_k} \right) \| w'_{\ep_k}(\tau) \| \, \text{d}\tau \nonumber\\
& \geq  \int_0^{\tau_k} a_{p,\ep}\left( \frac{w_{\ep_k}(\tau)}{\ep_k} \right) \| w'_{\ep_k}(\tau) \| \, \text{d}\tau \nonumber \\& \qquad  + \int_{\sigma_k}^1 a_{p,\ep}\left( \frac{w_{\ep_k}(\tau)}{\ep_k} \right) \| w'_{\ep_k}(\tau) \| \, \text{d}\tau \nonumber \\
&= d_{p,\ep}(\xi_1, w_{\ep_k}(\tau_k)) + d_{p,\ep}(w_{\ep_k}(\sigma_k), \xi_2), \label{res1}
\end{align}
since any geodesic curve is locally geodesic. By construction $\|w_{\ep_k}(\tau_k) - \xi_2\| \leq \sqrt{d}\ep_k$ and $w_{\ep_k}(\tau_k) \in \ep_k \Omega_w$. Set $\xi_{1,\ep_k} = \xi_1$ and $\xi_{2,\ep_k} = w_{\ep_k}(\tau_k)$ for all $k$. Applying Theorem \ref{3:lem:exi} it follows that
\begin{equation}\label{res2}
\lim_{k \rightarrow \infty} d_{p,\ep}(\xi_1, w_{\ep_k}(\tau_k))  = \lim_{k \rightarrow \infty} d_{\ep_k}(\xi_1, \xi_2)= M.
\end{equation}
The quantity $d_{p,\ep}(w_{\ep_k}(\sigma_k), \xi_2)$ is the distance between the centre of the ball $\ep_k (B_{\rho}(\xi_2) + (k,k,\cdots))$ and a point on its boundary. As the ball $\ep_k (B_{\rho}(\xi_2) + (k,k,\cdots))$ is contained in $\Omega_g$ it follows that
\begin{equation}
d_{p,\ep}(w_{\ep_k}(\sigma_k), \xi_2) = \frac{\beta}{\ep_k^p}\|w_{\ep_k}(\sigma_k) - \xi_2\| 
 = \beta \ep_k^{1-p} \rho. \label{res3}
\end{equation}
Combining \eqref{res1}, \eqref{res2} and \eqref{res3} then taking the limit $k \rightarrow \infty$ gives that
\begin{equation*}
\lim_{k \rightarrow \infty} F_{p,\ep_k}(w_{\ep_k}) \geq
\begin{cases}
\infty & \text{ if } p > 1,\\
M + \beta \rho & \text{ if } p = 1.
\end{cases}
\end{equation*}
Hence there exists sequences $(\tilde \ep_k)_{k=1}^{\infty}$ and $(\ep_k)_{k=1}^{\infty}$ converging to $0$ such that $M \geq M + \beta\rho$. Since $\beta, \rho > 0$ the result is shown.
\end{proof}

The geometrical assumption is not particularly restrictive. The assumption in Theorem \ref{3:t:65} would apply to the case where the inclusion is contained in the interior of the unit cell and the inclusion is at the centre of the unit cell. 

\subsection{The equivalence of $\G$-convergence and metric convergence for $p < 1$}

In this section we show that the boundary value problem $\G$-converges if and only if the induced metrics converge locally uniformly. This extends the theory of \cite{buttazzo01a} to a class of non-uniformly bounded two-phase Riemannian length functionals. That is, when the sequence $d_{\ep}$ fails to satisfy $\alpha \|\xi_2 - \xi_1\| \leq d_{\ep}(\xi_1,\xi_2) \leq \beta \|\xi_2 - \xi_1\|$ uniformly in $\ep$ for all $\xi_1,\xi_2 \in \mathbb{R}^d$.

The following lemma shows that we can improve the bounds on the induced metric so that $d_{p,\ep}$ is almost uniformly equivalent to the Euclidean metric.

\begin{Lemma}\label{3:bds:123}
Let $p < 1$, $\ep > 0$ and $\xi_1,\xi_2 \in \mathbb R^d$. Then there exists $C_1,C_2 > 0$ such that
\begin{equation*}
\|\xi_1 - \xi_2\| - C_1\ep \leq d_{p,\ep}(\xi_1, \xi_2) \leq \beta \|\xi_1 - \xi_2\| + C_2\ep^{1-p}.
\end{equation*}
\end{Lemma}

\begin{proof}
Let $u$ be a geodesic joining $\xi_1$ to $\xi_2$. By Lemma \ref{3:purpose} it follows that the set $T := \{ \tau \in (0,1) \colon u(\tau) \in \ep \left( \Omega_g + \mathbb Z^d \right)\}$ takes one of the forms
\begin{equation*}
\emptyset \text{ or } [0,\tau_1) \text{ or } (\tau_2,1] \text{ or } [0,\tau_1) \cup (\tau_2,1] \text{ or } [0,1]
\end{equation*}
for some $\tau_1, \tau_2 \in (0,1)$ with $\tau_1 < \tau_2$. Suppose first that $T = [0,\tau_1) \cup (\tau_2,1]$; the cases when $\tau_2 = 1$ or $\tau_1 = 0$ following in an identical fashion. It holds that
\begin{equation}
F_{p,\ep}(u) = d_{p,\ep}(\xi_1,u(\tau_1)) + d_{p,\ep}(u(\tau_1),u(\tau_2)) + d_{p,\ep}(u(\tau_2),\xi_2).\label{3:ctd}
\end{equation}
Observe that by construction $\|\xi_1 - u(\tau_1)\| \leq \sqrt{d}\ep$ and $\|\xi_2 - u(\tau_2)\| \leq \sqrt{d}\ep$, and hence by the growth condition \eqref{3:7:est} it follows that
\begin{equation*}
d_{p,\ep}(\xi_1,u(\tau_1)) \leq \beta \sqrt{d}\ep^{1-p} \text{ and } d_{p,\ep}(u(\tau_2),\xi_2) \leq \beta \sqrt{d}\ep^{1-p}.
\end{equation*}
Using Lemma \ref{3:purpose} and the fact that $u(\tau_1), u(\tau_2) \in \ep \Omega_w$, it follows that $a_{p,\ep}(u(\tau)) = a(u(\tau))$ for all $\tau \in [\tau_1,\tau_2]$ and hence $d_{p,\ep}(u(\tau_1),u(\tau_2)) = d_{\ep}(u(\tau_1),u(\tau_2))$. By the triangle inequality and \eqref{3:7:est} we have
\begin{align*}
d_{\ep}(u(\tau_1),u(\tau_2)) \leq d_{\ep}(\xi_1,\xi_2) + \beta\left(\|\xi_1 - u(\tau_1)\| +\|\xi_2 - u(\tau_2)\|   \right)\\
 \leq \beta\|\xi_2 - \xi_1\| + 2\beta \sqrt{d}\ep.\\
\end{align*}
Consequently 
\begin{equation*}
d_{p,\ep}(\xi_1, \xi_2) = F_{p,\ep}(u) \leq \beta\|\xi_2 - \xi_1\| + 2\beta \sqrt{d}\ep^{1-p} + 2\beta \sqrt{d}\ep.
\end{equation*}
Continuing from \eqref{3:ctd} and applying the triangle inequality with Lemma \ref{3:purpose}, we have that,
\begin{align*}
F_{p,\ep}(u)  \geq \int_{\tau_1}^{\tau_2} a_{p,\ep} \left( \frac{u(\tau)}{\ep} \right) \| u'(\tau) \| \, \text{d}\tau& = d_{p,\ep}(u(\tau_1),u(\tau_2))\\
& = d_{\ep}(u(\tau_1),u(\tau_2))\\
& \geq d_{\ep}(\xi_1,\xi_2) - d_{\ep}(\xi_1,u(\tau_1)) \\
& \qquad -  d_{\ep}(\xi_2,u(\tau_2)).
\end{align*}
From \eqref{3:7:est} it follows that
\begin{equation*}
F_{p,\ep}(u)  \geq \|\xi_2 - \xi_1\| - 2\beta\sqrt{d}\ep.
\end{equation*}
Hence the bounds are illustrated. The remaining case when $T=[0,1]$ follows in a similar manner. 
\end{proof}

The following lemma improves pointwise convergence to local uniform convergence as in the uniformly bounded case. The key is that we are still close to the uniformly bounded case, due to the improved growth bounds of Lemma \ref{3:bds:123}.

\begin{Lemma}\label{3:lem:56}
If the metrics $d_{p,\ep}$ converge pointwise to $d$ then they converge locally uniformly. 
\end{Lemma}

\begin{proof}
We follow the proof of \cite[Proposition 2.3]{buttazzo01a}. Take $(x,y) \in \mathbb R^d \times \mathbb R^d$ and let $(x_{\ep},y_{\ep})$ be a sequence converging to $(x,y)$. Then
\begin{align*}
\lim_{\ep \rightarrow 0} \left| d_{\ep}(x_{\ep},y_{\ep}) - \psi(y-x) \right| &\leq \lim_{\ep \rightarrow 0} \left| d_{\ep}(x_{\ep},y_{\ep}) - d_{\ep}(x,y) \right| \\
& \qquad + \lim_{\ep \rightarrow 0} \left| d_{\ep}(x,y) - \psi(y-x) \right|\\
&\leq \lim_{\ep \rightarrow 0} C \left(|x_{\ep} - x| + |y_{\ep} - y| + 2C\ep^{1-p} \right),
\end{align*}
using the bounds in Lemma \ref{3:bds:123} and the pointwise convergence of $d_{\ep}$. As the point $(x,y)$ and sequence $\{ (x_{\ep}, y_{\ep}) \}_{\ep > 0}$ are arbitrary, this implies the local uniform convergence required. 
\end{proof}

We are now in a position to prove one of our main homogenisation results, using a modification of the method in \cite[Theorem 3.1]{buttazzo01a}.

\begin{Theorem}\label{3:t:34}
If the induced metrics $d_{p,\ep}$ converge locally uniformly to a metric $d$ on $\mathbb R^d$, then the sequence of functionals $F_{p,\ep}$ defined on $\mathcal A (\xi_1,\xi_2)$ $\G$-converge with respect to the $L^{\infty}\left((0,1)\right)$ norm topology to 
\begin{equation*}
\int_0^1\psi(u'(\tau))\, \text{d}\tau,
\end{equation*}
for all $\xi_1, \xi_2 \in \mathbb R^d$. The function $\psi$ is given by
\begin{equation*}
\psi(\xi) = \lim_{\ep \rightarrow 0} \min_{u \in \mathscr A(0,\xi)} F_{\ep}(u).
\end{equation*}
\end{Theorem}

\begin{proof}
Fix $\xi_1, \xi_2 \in \mathbb{R}^d$. Let $(\ep_k)_{k=1}^{\infty} \subset (0,\infty)$ converge to zero. Fix $u \in \mathcal A(\xi_1,\xi_2)$ and let $u_{\ep_k} \in \mathcal A(\xi_1,\xi_2)$ converge to $u \in L^\infty(0,1)$ as $k \rightarrow \infty$, then $u_{\ep_k} \rightarrow u$ pointwise as $k \rightarrow \infty$. Let $\pi_N = \{ \tau_0 ,... ,\tau_N\}$ be a partition of $[0,1]$ such that $|\tau_j - \tau_{j+1}| = 1/N$ for $j = 1, ... ,N$. 
Then
\begin{align*}
F_{\ep_k}(u_{\ep_k}) &= \sum_{i=1}^{N} \int_{\tau_{i-1}}^{\tau_i} a_{p,\ep}\left(\frac{u_{\ep_k}(\tau)}{\ep_k} \right) \|u'_{\ep_k}(\tau)\| \, \text{d}\tau \\
& \geq \sum_{i=1}^{N} d_{p,\ep_k}(u_{\ep_k}(\tau_{i-1}),u_{\ep_k}(\tau_i)),
\end{align*}
using the invariance of the length functional under reparameterisations. Therefore applying the triangle inequality for the induced metric, and Lemma \ref{3:bds:123}, we obtain
\begin{align*}
F_{\ep_k}(u_{\ep_k}) &\geq \sum_{i=1}^{N} d_{p,\ep_k}\left(u(\tau_{i-1}),u(\tau_i) \right) -  d_{p,\ep_k}\left(u_{\ep_k}(\tau_{i}),u(\tau_i) \right) \\ & \qquad \quad - d_{p,\ep_k}\left(u(\tau_{i-1}),u_{\ep_k}(\tau_{i-1}) \right)\\
&\geq \sum_{i=1}^{N} d_{p,\ep_k}\left(u(\tau_{i-1}),u(\tau_i) \right) - \beta \| u_{\ep_k}(\tau_{i}) - u(\tau_i)\| \\
& \qquad -  \beta\|u(\tau_{i-1}) - u_{\ep_k}(\tau_{i-1}) \| - 2C\ep_k^{1-p}.
\end{align*}
Taking the limit as $k \rightarrow \infty$ and using the fact that $d_{p,\ep}$ converges pointwise by Theorem \ref{3:lem:exi}, gives that
\begin{equation*}
\liminf_{k \rightarrow \infty} F_{\ep_k}(u_{\ep_k}) \geq \sum_{i=1}^{N} \psi\left(u(\tau_i)-u(\tau_{i-1}) \right).
\end{equation*}
Using the $1-$homogeneity of $\psi$, c.f. \cite{braides02b}, gives
\begin{equation*}
\liminf_{k \rightarrow \infty} F_{\ep_k}(u_{\ep_k}) \geq \sum_{i=1}^{N} \psi\left(\frac{u(\tau_i)-u(\tau_{i-1})}{|\tau_i-\tau_{i-1}|} \right)|\tau_i-\tau_{i-1}| = \int_0^1 \psi(u_N'(\tau)) \, \text{d}\tau,
\end{equation*}
where $u_N$ is the linear interpolation of $u$ on $\pi_N$. Sending $N \rightarrow \infty$, and applying the dominated convergence Theorem, to prove the lim-inf inequality.

We now show the lim-sup inequality. Fix $u \in \mathcal A(\xi_1, \xi_2)$. Then choosing a sequence $(M_k)_{k=1}^{\infty} \subset \mathbb N$ such that
\begin{equation*}
\lim_{k \rightarrow \infty} M_k \sup_{\xi_1, \xi_2 \in K} |d_{\ep_k}(\xi_1,\xi_2) - \psi(\xi_2 - \xi_1)| = 0,
\end{equation*}
where $K \subset \subset \mathbb R^d$ such that $\text{graph}(u) \subset \text{int}(K)$. Let $\pi_{M_k} = \{ \tau_0 ,... ,\tau_{M_k}\}$ be a partition of $[0,1]$ such that $|\tau_j - \tau_{j+1}| = 1/M_k$ for $j = 1, ... ,M_k$. Define the function $u_{\ep_k}$ by
\begin{multline}\label{123456789}
u_{\ep_k}(\tau) = u(\tau) +\\ \text{argmin}_{w \in W^{1,\infty}_0\left((0,1)\right)} \int_{\tau_{i-1}}^{\tau_i} a_{p,\ep_k}\left(\frac{u(\tau)+w(\tau)}{\ep}\right) \|u'(\tau) + w'(\tau) \| \, \text{d}\tau
\end{multline}
for $\tau \in [{\tau_{i-1}},{\tau_i}]$ in the partition $\pi_{M_k}$. Clearly, by construction, $u_{\ep_k} \in \mathcal A(\xi_1, \xi_2)$. Fix $k$ and $t \in [0,1]$ and suppose that $\tau \in [\tau_{i-1},\tau_i]$. Then
\begin{equation}
\left\| u_{\ep_k}(\tau) - u(\tau) \right\| \leq \left\| u_{\ep_k}(\tau) - u(\tau_{i-1}) \right\| + \left\|u(\tau_{i-1}) - u(\tau) \right\|.
\end{equation}
By Lemma \ref{3:bds:123} it holds that
\begin{align*}
\left\| u_{\ep_k}(\tau) - u(\tau_{i-1}) \right\| &\leq d_{p,\ep_k}( u_{\ep_k}(\tau),u(\tau_{i-1}) ) + C\ep_k^{1-p},\\
&\leq d_{p,\ep_k}( u_{\ep_k}(\tau_i),u(\tau_{i-1}) )+C\ep_k^{1-p},\\
&\leq \beta \left\| u_{\ep_k}(\tau_i) - u_{\ep_k}(\tau_{i-1})  \right\| + 2C\ep_k^{1-p}, \\
& = \beta \left\| u(\tau_i) - u(\tau_{i-1})  \right\| + 2C\ep_k^{1-p}.
\end{align*}
Therefore, letting the Lipschitz constant of $u$ be denoted by $\Lambda$, 
\begin{align*}
\left\| u_{\ep_k}(\tau) - u(\tau) \right\| &\leq  \beta \left\| u(\tau_i) - u(\tau_{i-1})  \right\| + \left\|u(\tau_{i-1}) - u(\tau) \right\|+ 2C\ep_k^{1-p}\\
&\leq  \beta \Lambda \left|\tau_i - \tau_{i-1}  \right| + \Lambda\left|\tau_{i-1} - \tau \right|+ 2C\ep_k^{1-p}\\
&\leq  \frac{(1+\beta)\Lambda}{M_k}  + 2C\ep_k^{1-p},
\end{align*}
using the fact that $\left|\tau_i - \tau_{i-1}  \right| = 1/M_k$ for all $k$. Hence $u_{\ep_k} \rightarrow u$ in $L^{\infty}\left((0,1)\right)$. It remains to show that $u_{\ep_k}$ has the desired properties. Observe that
\begin{align}
\int_0^1\psi(u'(\tau)) \, \text{d}\tau & \geq \sum_{i=1}^{M_k} \psi(u(\tau_i)-u(\tau_{i-1})) \nonumber \\
& \geq \sum_{i=1}^{M_k}  d_{p,\ep}(u_{\ep_k}(\tau_i),u_{\ep_k}(\tau_{i-1})) -\nonumber \\
& \qquad \sum_{i=1}^{M_k} \left| \psi(u(\tau_i)-u(\tau_{i-1})) - d_{p,\ep}(u_{\ep_k}(\tau_i),u_{\ep_k}(\tau_{i-1})) \right|.\label{3:th1}
\end{align}
By construction it holds that
\begin{equation*}
 \sum_{i=1}^{M_k}  d_{p,\ep}(u_{\ep_k}(\tau_i),u_{\ep_k}(\tau_{i-1})) = F_{p,\ep}(u_{\ep_k}),
\end{equation*}
furthermore, for $\ep$ sufficiently small,
\begin{multline*}
\sum_{i=1}^{M_k} \left| \psi(u(\tau_i)-u(\tau_{i-1})) - d_{p,\ep}(u_{\ep_k}(\tau_i),u_{\ep_k}(\tau_{i-1})) \right| \\ \leq M_k \sup_{\xi_1, \xi_2 \in K} |d_{\ep_k}(\xi_1,\xi_2) - \psi(\xi_2 - \xi_1)|.
\end{multline*}
Hence by \eqref{3:th1}, the choice of $(M_k)_{k=1}^{\infty}$ and the liminf inequality,
\begin{equation*}
\int_0^1\psi(u'(\tau)) \, \text{d}\tau \geq \limsup_{k \rightarrow \infty} F_{p,\ep}(u_{\ep_k}) \geq \liminf_{k \rightarrow \infty} F_{p,\ep}(u_{\ep_k}) \geq \int_0^1\psi(u'(\tau)) \, \text{d}\tau.
\end{equation*}
Since the choice of the sequence $(\ep_k)_{k=1}^{\infty}$ was arbitrary and the limit is independent of this choice, and $\G$-convergence follows.
\end{proof}

It remains to prove the converse. 

\begin{Theorem}\label{3:t:390}
If the sequence of functionals $F_{p,\ep}$ defined on $\mathcal A (\xi_1,\xi_2)$ $\G$-converge with respect to the $L^{\infty}\left((0,1)\right)$ norm topology to 
\begin{equation*}
\int_0^1\psi(u'(\tau))\, \text{d}\tau,
\end{equation*}
for all $\xi_1, \xi_2 \in \mathbb R^d$ then the induced metrics converge locally uniformly to a norm on $\mathbb R^d$.
\end{Theorem}

\begin{proof}
Fix $\xi_1, \xi_2 \in \mathbb R^d$. Applying the fundamental theorem of $\G$-convergence \cite{braides02a,braides98a} it follows that the limit
\begin{equation*}
\lim_{\ep \rightarrow 0} d_{p,\ep}(\xi_1,\xi_2) = \lim_{\ep \rightarrow 0} \min_{u \in \mathscr A(\xi_1,\xi_2)} F_{p,\ep}(u)
\end{equation*}
exists, and hence $d_{p,\ep}$ converges pointwise to $\psi$. By Lemma \ref{3:lem:56} the local uniform convergence follows.
\end{proof}

\bibliographystyle{abbrv}
\bibliography{references}

\end{document}